\documentclass[a4paper,11pt,reqno]{amsart}

\usepackage{enumerate}
\usepackage[PS]{diagrams}

\newtheorem{theorem}{Theorem}[section]
\newtheorem{corollary}[theorem]{Corollary}
\newtheorem{proposition}[theorem]{Proposition}
\newtheorem{lemma}[theorem]{Lemma}

\theoremstyle{definition}
\newtheorem{definition}[theorem]{Definition}

\let\tensor=\otimes
\let\nbd=\nobreakdash
\let\iso=\cong

\newcommand{\id}{\mathrm{id}}
\newcommand{\bZ}{\ensuremath{\mathbb{Z}}}

\newcommand{\Ch}{\ensuremath{\mathrm{Ch}}}
\newcommand{\Chb}{\ensuremath{\mathrm{Ch}^\mathrm{b}}}

\newcommand{\Tot}{\ensuremath{\mathop{\mathrm{Tot}}}}
\newcommand{\coker}{\ensuremath{\mathrm{coker}\,}}

\newcommand{\resp}{{\it resp.}}
\newcommand{\ie}{{\it i.e.}}

\begin{document}

\title{Finite domination and Novikov rings. Iterative approach}

\date{08.09.2011, minor corrections 12.10.2011}

\author{Thomas H\"uttemann}

\address{Thomas H\"uttemann\\ Queen's University Belfast\\ School of
  Mathematics and Physics\\ Pure Mathematics Research Centre\\ Belfast
  BT7~1NN\\ Northern Ireland\\ UK}

\email{t.huettemann@qub.ac.uk}

\urladdr{http://huettemann.zzl.org/}

\author{David Quinn}

\address{David Quinn\\ Queen's University Belfast\\ School of
  Mathematics and Physics\\ Pure Mathematics Research Centre\\ Belfast
  BT7~1NN\\ Northern Ireland\\ UK}

\email{david.quinn@qub.ac.uk}


\thanks{This work was supported by the Engineering and Physical
  Sciences Research Council [grant number EP/H018743/1].}

\begin{abstract}
  Suppose $C$ is a bounded chain complex of finitely generated free
  modules over the \textsc{Laurent} polynomial ring $L =
  R[x,x^{-1}]$. Then $C$ is $R$\nbd-finitely dominated, \ie, homotopy
  equivalent over~$R$ to a bounded chain complex of finitely generated
  projective $R$\nbd-modules, if and only if the two chain complexes
  $C \tensor_L R((x))$ and $C \tensor_L R((x^{-1}))$ are acyclic, as
  has been proved by \textsc{Ranicki}. Here $R((x)) = R[[x]][x^{-1}]$
  and $R((x^{-1})) = R[[x^{-1}]][x]$ are rings of formal
  \textsc{Laurent} series, also known as \textsc{Novikov} rings. In
  this paper, we prove a generalisation of this criterion which allows
  us to detect finite domination of bounded below chain complexes of
  projective modules over \textsc{Laurent} rings in several
  indeterminates.
\end{abstract}

\maketitle

Finiteness conditions for chain complexes of modules play an important
role in both algebra and topology. For example, given a group~$G$ one
might ask whether the trivial $G$\nbd-module~$\bZ$ admits a resolution
by finitely generated projective $\bZ[G]$\nbd-modules; existence of
such resolutions is relevant for the study of group homology of~$G$,
and has applications in the theory of duality groups \cite{Brown}. For
topologists, finite domination of chain complexes is related, among
other things, to questions about finiteness of $CW$ complexes, the
topology of ends of manifolds, and obstructions for the existence of
non-singular closed $1$\nbd-forms \cite{Ranicki-findom, Schuetz}.

A chain complex~$C$ of $R[z,z^{-1}]$\nbd-modules is called {\it
  finitely dominated} if it is homotopy equivalent, as a complex of
$R$\nbd-modules, to a bounded complex of finitely generated projective
$R$\nbd-modules. Finite domination of~$C$ can be characterised in
various ways; \textsc{Brown} considered compatibility of the functors
$M \mapsto H_* (C;M)$ and $M \mapsto H^* (C;M)$ with products and
direct limits, respectively \cite[Theorem~1]{Brown}, while
\textsc{Ranicki} showed that $C$ is finitely dominated if and only if
the \textsc{Novikov} homology of~$C$ is trivial (see
\cite[Theorem~2]{Ranicki-findom}, and Theorem~\ref{thm:ranicki}
below).

In this paper we consider finite domination of chain complexes over a
\textsc{Laurent} polynomial ring~$L$ with several indeterminates. In
Theorem~\ref{thm:main} we give a complete characterisation of finitely
dominated chain complexes in terms of their \textsc{Novikov} homology
over subrings of~$L$ generated by a subset of the indeterminates.

Related results have been discussed by \textsc{Sch\"utz}
\cite[\S4]{Schuetz}, but note that the criterion given there involves
infinitely many trivial \textsc{Novikov} homology modules, whereas our
result utilises \textsc{Novikov} homology with respect to finitely
many rings only.

In \S1 we introduce the notion of a finitely dominated chain complex,
and formulate our main result. In \S2 we review some constructions
from homological algebra and discuss the algebraic mapping torus of a
self map of a chain complex. Then Theorem~\ref{thm:main} is proved
in~\S3. We finish the paper by giving an explicit example of a
non-trivial finitely dominated chain complex in~\S4, and by discussing
finite domination over a field in~\S5.

\section{Finitely dominated chain complexes}

\label{sec:finit-domin-chain}

Let $A$~denote a ring with unit. We write $\Ch(A)$ for the category of
chain complexes of (right) $A$\nbd-modules, and $\Chb(A)$ for the full
subcategory of bounded chain complexes.

\begin{definition}
  Let $S$~be a subring of~$A$; every chain complex of $A$\nbd-modules
  is then, by restriction, also a chain complex of $S$\nbd-modules. We
  say that the chain complex $C \in \Ch(A)$ is
  \begin{enumerate}[{\rm (a)}]
  \item {\it $S$\nbd-finite} if it is bounded and consists of finitely
    generated free $S$-modules;
  \item {\it homotopy $S$\nbd-finite} if it is homotopy equivalent to
    an $S$\nbd-finite complex $D \in \Chb(S)$;
  \item {\it strict $S$\nbd-perfect} if it is bounded and consists of
    finitely generated projective $S$\nbd-modules;
  \item {\it $S$\nbd-finitely dominated} if it is homotopy equivalent
    to a strict $S$\nbd-perfect complex $D \in \Chb(S)$.
  \end{enumerate}
\end{definition}

\noindent Given an $S$\nbd-finitely dominated complex $C \in \Ch(A)$
there exists a strict $S$\nbd-perfect complex $D \in \Ch(S)$ homotopy
equivalent to~$C$. The {\it finiteness obstruction of~$C$} is defined
to be
\[\chi (C) = \sum_{j \in \bZ} (-1)^j [D_j] \in \tilde K_0 (S) \ ;\] it
is independent of the choice of~$D$. {\it The complex~$C$ is homotopy
  $S$\nbd-finite if and only if its finiteness obstruction is
  trivial\/}; see \cite[Theorem~1.7.12]{Rosenberg} for a textbook
proof. In this sense, algebraic $K$\nbd-theory detects homotopy
finiteness of finitely dominated chain complexes.

To find out whether a given complex $C \in \Ch(A)$ is homotopy
$S$\nbd-finite one should thus first determine whether it is
$S$\nbd-finitely dominated. In the special case $S=R$ and $A =
R[x,x^{-1}]$, \textsc{Ranicki} has given the following homological
characterisation:

\begin{theorem}[{\textsc{Ranicki} \cite[Theorem~2]{Ranicki-findom}}]
  \label{thm:ranicki}
  Let $C$ be a bounded chain complex of finitely generated free
  $R[x,x^{-1}]$\nbd-modules. The following conditions are equivalent:
  \begin{enumerate}[{\rm (a)}]
  \item \label{item:R_fd} The complex~$C$ is $R$\nbd-finitely dominated.
  \item \label{item:R_acyclic} Both the following chain complexes are acyclic:
    \[C \tensor_{R[x,x^{-1}]} R((x)) \quad \text{and} \quad C
    \tensor_{R[x,x^{-1}]} R((x^{-1})) \ .\]
  \end{enumerate}
\end{theorem}

\noindent Here we denote by $R[[x]]$ the ring of formal power series
in the indeterminate~$x$, and write $R((x))$ for the localisation
of~$R[[x]]$ by~$x$. That is, $R((x))$ is the ring of formal
\textsc{Laurent} series
\[\sum_{j=k}^\infty a_j x^j \ , \qquad k \in \bZ \ ,\] also known as
the \textsc{Novikov} ring of~$R$ in~$x$. Similarly, $R[[x^{-1}]]$ is
the ring of formal power series in the indeterminate~$x^{-1}$, and the
\textsc{Novikov} ring $R((x^{-1}))$ is its localisation
by~$x^{-1}$. Elements of the latter can be written as formal
\textsc{Laurent} series of the type
\[\sum_{j=-\infty}^k a_j x^j \ , \qquad k \in \bZ \ .\]

\medbreak

As it stands this result is not adapted to iteration. In more detail,
suppose that $R$~itself is a \textsc{Laurent} ring $R = K[y,y^{-1}]$,
over some ring~$K$; one would want then to be able to apply
\textsc{Ranicki}'s theorem twice: first to $R \subset R[x,x^{-1}]$,
then to $K \subset K[y,y^{-1}] = R$. One difficulty here is that the
first application leaves us with a chain complex which consists of
projective rather than free modules. In addition, the \textsc{Laurent}
variables are dealt with in a specific order which, intuitively
speaking, should have no bearing on the question of finite
domination. Both issues are addressed in our main result below.

\medbreak

Write $R_n$~for the ring of \textsc{Laurent} polynomials in
$n$~indeterminates with coefficients in~$R$,
\[R_n = R[x_1,\, x_1^{-1},\, x_2,\, x_2^{-1},\, \cdots,\, x_n,\,
x_n^{-1}] \ ,\] so that $R_0 = R$ and $R_k = R_{k-1}[x_k,x_k^{-1}]$
for $k \geq 1$. We will prove the following generalisation of
Theorem~\ref{thm:ranicki} to many variables:

\begin{theorem}
  \label{thm:main}
  Let $n \geq 1$. For a bounded below complex~$C$ of projective
  $R_n$\nbd-modules (not necessarily finitely generated) the following
  four conditions are equivalent:
  \begin{enumerate}[{\rm (a)}]
  \item \label{item:fd} The complex~$C$ is $R$\nbd-finitely dominated.
  \item \label{item:hf} The complex~$C$ is $R$\nbd-finitely dominated,
    and for all $n!$ re-numberings of the variables $x_1,\, x_2,\,
    \cdots,\, x_n$, the complex~$C$ is homotopy $R_j$\nbd-finite for
    $j = 1,\, 2,\, \cdots,\, n$.
  \item \label{item:acyclicall} $C$ is $R_n$\nbd-finitely dominated,
    and for all $n!$ re-numberings of the variables $x_1,\, x_2,\,
    \cdots,\, x_n$ the following chain complexes are acyclic:
    \[C \tensor_{R_j} R_{j-1}((x_j)) \quad \text{and} \quad C
    \tensor_{R_j} R_{j-1}((x_j^{-1})) \ , \quad 1 \leq j \leq n \ .\]
  \item \label{item:acyclicfixed} $C$ is $R_n$\nbd-finitely dominated,
    and for some re-numbering of the variables $x_1,\, x_2,\, \cdots,\,
    x_n$ the following chain complexes are acyclic:
    \[C \tensor_{R_j} R_{j-1}((x_j)) \quad \text{and} \quad C
    \tensor_{R_j} R_{j-1}((x_j^{-1})) \ , \quad 1 \leq j \leq n \ .\]
  \end{enumerate}
\end{theorem}

Note that this theorem says in particular that an $R$\nbd-finitely
dominated chain complex of $R_n$\nbd-modules is automatically homotopy
equivalent over~$R_k$, $1 \leq k \leq n$, to an $R_k$\nbd-finite
complex consisting of free rather than projective
modules. Nevertheless the proof forces us to work with chain complexes
of modules which {\it a priori} consist of projective modules.

\medbreak

We start by fixing our sign conventions for some constructions from
homological algebra, together with a collection of standard results
which will be used repeatedly in the sequel. We then develop the
relevant theory of mapping tori, and apply all this in the proof of
the main theorem. We finish the paper by giving a concrete non-trivial
example of a finitely dominated chain complex over a \textsc{Laurent}
ring in finitely many indeterminates, and by discussing finite
domination over fields, which essentially reduces to an exercise in
linear algebra.

The methods used here borrow heavily from those of \textsc{Ranicki}
\cite{Ranicki-findom}, modified to allow for the presence of several
indeterminates and non-free modules. It is possible to approach finite
domination over \textsc{Laurent} rings in several indeterminates from
the point of view of toric geometry; this perspective yields a
completely different set of conditions, and will be presented in a
forthcoming paper.

\section{Mapping cones and mapping tori}
\label{sec:algebr-mapp-torus}

\subsection*{Chain complexes and mapping cones}

We begin with listing some conventions. We will consider arbitrary
chain complexes of (right) modules over some ring with unit~$A$; we
think of chain complexes as being ``vertical''. The $k$th suspension
($k \in \bZ$) of a chain complex~$C$ is the chain complex $C[k]$
defined by $C[k]_\ell = C_{\ell - k}$ with differential changed by the
sign $(-1)^k$.

A twofold chain complex is a chain complex in the category of chain
complexes, that is, a family $(D_{p,q})_{p,q \in \bZ}$ of
$R$\nbd-modules together with ``horizontal'' and ``vertical''
differential
\[\partial_h\colon D_{p,q} \rTo D_{p-1,q} \quad \text{and}
\quad \partial_v \colon D_{p,q} \rTo D_{p,q-1}\] satisfying
$\partial_h^2 = 0$, $\partial_v^2=0$ and $\partial_h \partial_v
= \partial_v \partial_h$. The {\it total complex\/} of the twofold
chain complex~$D$ is a chain complex $\Tot(D)$. In chain degree~$n$ we
have, by definition,
\[\Tot(D)_n = \bigoplus_{p+q = n} D_{p,q} \ ,\] and the differential
is induced by
\[\partial_h \colon D_{p,q} \rTo D_{p-1,q} \quad \text{and}
\quad (-1)^p \partial_v \colon D_{p,q} \rTo D_{p,q-1} \ .\]

A map of chain complexes $f \colon C \rTo B$ can be considered as a
twofold chain complex with $B$ in column $p=0$ and $C$ in column
$p=1$, and horizontal differential given by~$f$. Its total complex is
known as the {\it mapping cone of~$f$}, denoted $\mathrm{Cone}\,
(f)$. We have $\big( \mathrm{Cone} (f) \big)_k = C_{k-1} \oplus
B_k$. There is a natural long exact homology sequence associated to
this construction:
\begin{equation}
  \label{eq:cone_les}
  \ldots \rTo^f H_k B \rTo H_k \mathrm{Cone}(f) \rTo H_{k-1} C \rTo^f
  H_{k-1} B \rTo \ldots
\end{equation}
In particular, application of the Five Lemma shows that the mapping
cone construction is invariant under quasi-isomorphism of maps of
chain complexes. That is, given a commutative diagram of chain
complexes
\begin{diagram}
  B & \rTo^f & C \\ \dTo<\simeq && \dTo<\simeq \\ D & \rTo^g & E
\end{diagram}
where the vertical maps are quasi-isomorphisms, the induced map
\[\mathrm{Cone}\, (f) \rTo \mathrm{Cone}\, (g)\] is a
quasi-isomorphism as well. --- Let $f \colon C \rTo B$ be a map of
chain complexes as before.  The canonical projection from the
$B$\nbd-summands assemble to a natural map $\mathrm{Cone}\,(f) \rTo
\coker(f)$.

\begin{lemma}
  \label{lem:cone_coker}
  If $f \colon C \rTo B$ is an injective map of chain complexes, the
  natural map $\mathrm{Cone}\,(f) \rTo \coker(f)$ is a
  quasi-isomorphism.
\end{lemma}

\begin{proof}
  The long exact sequence in~\eqref{eq:cone_les} and the long exact
  sequence associated to the short exact sequence
  \[0 \rTo C \rTo^f B \rTo \coker(f) \rTo 0\] assemble into a
  commutative ladder diagram, with two out of three maps the
  identity. By the Five Lemma, the remaining maps (which are induced
  by the map under investigation) are isomorphisms.
\end{proof}

We have defined the mapping cone by totalising a twofold chain
complex. Conversely, one can describe totalisation by iterating the
mapping cone construction. For us, the following special case will be
sufficient:

\begin{lemma}
  \label{lem:double_cone}
  Suppose we have maps of chain complexes $f \colon C \rTo B$ and $g
  \colon B \rTo A$ with $gf=0$. Let $D$~denote the twofold chain
  complex having $C$, $B$ and~$A$ in columns $2$, $1$ and~$0$, with
  horizontal differential given by $f$ and~$g$.  The map~$f$ induces
  an inclusion $C[1] \rTo \mathrm{Cone}\, (g)$, and we have an
  equality of chain complexes $\mathrm{Cone}\, \big( C[1] \rTo
  \mathrm{Cone}\, (g) \big) = \Tot (D)$.  \qed
\end{lemma}

\begin{corollary}
  \label{cor:ker_cone}
  Suppose that $0 \rTo C \rTo^f B \rTo^g A \rTo 0$ is a short exact
  sequence of chain complexes. Then there is a quasi-isomorphism
  \[C \rTo \big( \mathrm{Cone}\, (g) \big) [-1] \ .\]
\end{corollary}

\begin{proof}
  By the previous Lemma we have a map $\mu \colon C[1] \rTo
  \mathrm{Cone}\, (g)$, and this map is a quasi-isomorphism if and only
  if its mapping cone is acyclic. But its mapping cone is $\Tot (D)$,
  using the notation of that Lemma. There is a convergent spectral
  sequence
  \[E^1_{p,q} = H_q D_{*,p} \Longrightarrow H_{q+p} \Tot (D) \ ,\]
  cf.~\cite[\S XI.6]{MacLane-homology}; by exactness, its
  $E^1$\nbd-term is trivial, hence $\Tot(D)$~is acyclic. It follows
  that $\mu[-1] \colon C \rTo \big( \mathrm{Cone}\, (g) \big) [-1]$ is
  a quasi-isomorphism.
\end{proof}

\begin{proposition}
  \label{cone-no-obstruction}
  Suppose $C$ is an $R$\nbd-finitely dominated complex of projective
  $R$\nbd-modules. Then for any self map $f \colon C \rTo C$ the
  complex $\mathrm{Cone}\, (f)$ is homotopy $R$\nbd-finite.
\end{proposition}

\begin{proof}
  It is enough to show that the finiteness obstruction of~$C$ in
  $\tilde K_0(R)$ vanishes: since $K$-theory doesn't detect
  differentials, we have
  \[[\mathrm{Cone}\,(f)] = [C[1] \oplus C] = -[C] + [C] = 0 \in \tilde
  K_0 (R)\ .\]

  If $C$ is strict perfect one can easily give an explicit proof: for
  each $C_n$ choose a finitely generated projective module~$D_n$ such
  that $C_n \oplus D_n$ is free; choose $D_n = 0$ if $C_n = 0$. Then
  attaching the contractible two-step chain complexes $D_n \rTo^= D_n$
  (concentrated in degrees $n+1$ and~$n$) to $\mathrm{Cone}\,(f)$
  results in a bounded chain complex of finitely generated free
  $R$\nbd-modules which is homotopy equivalent, via the projection,
  to~$\mathrm{Cone}\,(f)$.
\end{proof}

\goodbreak

\subsection*{Algebraic mapping tori}

\begin{definition}
  \label{def:algebr-mapp-torus}
  Let $C$ be an arbitrary $R$\nbd-module chain complex, and let $h
  \colon C \rTo C$ be any chain map. The {\it algebraic mapping
    torus~$T(h)$ of~$h$} is defined as
  \[T(h) = \mathrm{Cone}\, \big( C \tensor_R R[x,x^{-1}]
  \rTo[l>=5em]^{h \tensor 1 - 1 \tensor x} C \tensor_R R[x,x^{-1}]
  \big) \ .\] Here the map ``$x$'' is given by the multiplication
  action of the indeterminate~$x$ on~$R[x,x^{-1}]$.
\end{definition}

By construction, $T(h)$~is an $R[x,x^{-1}]$\nbd-module chain complex
which is bounded if $C$~is bounded. If $C$ consists of finitely
generated (\resp{} projective, \resp{} free) $R$\nbd-modules, then
$T(h)$ consists of finitely generated (\resp{} projective, \resp{}
free) $R[x,x^{-1}]$\nbd-modules.

The mapping torus construction is functorial on the category of self
maps of $R$\nbd-module chain complexes in the following sense: a
commutative diagram
\begin{diagram}[LaTeXeqno]
  \label{diag:square}
  C & \rTo^f & C \\ \dTo<\alpha && \dTo<\alpha \\ D & \rTo^g & D
\end{diagram}
induces an $R[x,x^{-1}]$\nbd-linear chain map $\alpha_* \colon T(f)
\rTo T(g)$, and this assignment is compatible with vertical
composition (vertical stacking of square diagrams). Moreover, if
$\alpha$ is a quasi-isomorphism then so is~$\alpha_*$. Indeed, the
long exact sequences of mapping cones yield a commutative ladder
diagram {\small
\begin{diagram}[small]
  \cdots & H_{n+1} T(f) & \rTo & H_n \big(C \tensor_R R[x,x^{-1}]\big) & \rTo^\eta & H_n \big(C \tensor_R R[x,x^{-1}]\big) & \rTo & H_n T(f) & \cdots \\
  & \dTo>{\alpha_*} && \dTo>\alpha && \dTo<\beta && \dTo<{\alpha_*} \\
  \cdots & H_{n+1} T(g) & \rTo & H_n \big(D \tensor_R R[x,x^{-1}]\big) & \rTo^\zeta & H_n \big(D \tensor_R R[x,x^{-1}]\big) & \rTo & H_n T(g) & \cdots \\
\end{diagram}}%
(where $\eta = f \tensor 1 - 1 \tensor x$ and $\zeta = g \tensor 1 - 1
\tensor x$) with exact rows; since $R[x,x^{-1}]$~is a free
$R$\nbd-module, the two middle vertical maps are isomorphisms. It
follows from the Five Lemma that $\alpha_*$ is a quasi-isomorphism as
claimed.

\begin{lemma}
  \label{lem:basic_prop}
  \begin{enumerate}[{\rm (1)}]
  \item\label{item:1} Let $h \colon C \rTo C$ be a self map of an
    arbitrary chain complex~$C$ of $R$\nbd-modules. The map $h_*
    \colon T(h) \rTo T(h)$ is chain homotopic to~$x$, the
    ``multiplication by~$x$'' map. In particular, $h_*$ is a
    quasi-isomorphism.
  \item\label{item:2} Let $g,h \colon C \rTo C$ be homotopic chain
    maps. Then the mapping tori $T(g)$ and~$T(h)$ are isomorphic.
  \end{enumerate}
\end{lemma}

\begin{proof}
  (1) The homotopy is essentially given by projection on the second
  summand followed by inclusion into the first summand,
  \begin{multline*}
    T(h)_n = C_{n-1} \tensor_R R[x,x^{-1}] \, \oplus\,  C_n \tensor_R
    R[x,x^{-1}] \\
    \ \rTo[l>=4em]^{(\textrm{pr}_2, 0)} C_n \tensor_R
    R[x,x^{-1}] \, \oplus\, C_{n+1} \tensor_R R[x,x^{-1}] = T(h)_{n+1}\ .
  \end{multline*}
  The map~$x$ is an isomorphism, hence $h_*$ is a
  quasi-isomorphism.

  (2) Choose a chain homotopy $A \colon h \simeq g$ such that
  $\partial^CA + A \partial^C = h-g$, where $\partial^C$ is the
  differential of~$C$. Then it is easy to check by a straightforward
  computation that
  \begin{multline*}
    \begin{pmatrix}
      \id \tensor \id & 0 \\ A\tensor\id  & \id\tensor\id
    \end{pmatrix}
    \colon T(h)_n = C_{n-1} \tensor_R R[x,x^{-1}] \, \oplus\,  C_n \tensor_R
    R[x,x^{-1}] \\
    \rTo C_{n-1} \tensor_R R[x,x^{-1}] \, \oplus\,  C_n \tensor_R
    R[x,x^{-1}] = T(g)_n
  \end{multline*}
  defines a chain map with inverse given by the matrix
  $\left(\begin{smallmatrix} \id \tensor \id & 0 \\ -A\tensor\id &
      \id\tensor\id
  \end{smallmatrix}\right)$.
\end{proof}

\begin{proposition}[\textsc{Mather}'s mapping torus trick]
  \label{prop:weak_Mather}
  Suppose $f \colon C \rTo D $ and $g \colon D \rTo C$ are chain maps
  of $R$\nbd-module chain complexes. Then the two maps 
  \[f_* \colon T(gf) \rTo T(fg) \quad \text{and} \quad g_* \colon
  T(fg) \rTo T(gf)\] are quasi-isomorphisms. If both $C$ and~$D$ are
  bounded below complexes of projective modules, both maps are
  homotopy equivalences.
\end{proposition}

\begin{proof}
  The composition $g_* \circ f_* = (gf)_* \colon T(gf) \rTo T(gf)$ is
  a quasi-isomorphism by part~(1) of the previous Lemma; consequently,
  $f_*$ induces an injective map on homology, and $g_*$ induces a
  surjective map on homology. Swapping the r\^oles of $f$ and~$g$
  proves the claim.
\end{proof}

\begin{lemma}
  \label{lem:C_is_Tx}
  Let $C$ be a chain complex of $R[x,x^{-1}]$\nbd-modules (possibly
  unbounded). Then there is an $R[x,x^{-1}]$\nbd-linear
  quasi-isomorphism $T(x) \rTo C$ where $x$~is short for the
  $R$\nbd-module chain self map of~$C$ given by ``multiplication
  by~$x$''. The quasi-isomorphism is natural in~$C$.
\end{lemma}

\begin{proof}
  First we claim that for any $R[x,x^{-1}]$\nbd-module~$M$ there is an
  exact sequence of $R[x,x^{-1}]$\nbd-modules
  \begin{equation}
    \label{eq:ses}
    0 \rTo M \tensor_R R[x,x^{-1}] \rTo^{x \tensor 1 - 1 \tensor x} M \tensor_R R[x,x^{-1}] \rTo^\epsilon M \rTo 0 \ .
  \end{equation}
  Here the map denoted~$\epsilon$ is given by $m \tensor p \mapsto
  mp$. To begin with, $x \tensor 1 - 1 \tensor x$ is injective and
  $\epsilon$ is surjective, so it remains to prove exactness in the
  middle. First,
  \[\epsilon \circ (x \tensor 1 - 1 \tensor x)(m \tensor p) =
  \epsilon(mx \tensor p - m \tensor px) = mxp - mpx = 0\] since $x$~is
  in the centre of~$R[x,x^{-1}]$. This shows $\mathrm{Im}\, (x \tensor
  1 - 1 \tensor x) \subseteq \ker \epsilon$. We will prove the
  converse inclusion in a slightly indirect manner. We can consider
  the sequence~\eqref{eq:ses} as a sequence of $R$\nbd-modules and
  check exactness in the middle in the category of
  $R$\nbd-modules. The point is that $\epsilon$~has an $R$\nbd-linear
  section~$\sigma$ given by $m \mapsto m \tensor 1$. Consequently,
  there is an isomorphism $M \tensor_R R[x,x^{-1}] \iso \ker \epsilon
  \oplus \mathrm{Im}\, \sigma$ of $R$\nbd-modules, and every element
  in $\ker \epsilon$ is of the form $m - \sigma\epsilon m$, for some
  $m \in M \tensor_R R[x,x^{-1}]$. We can write~$m$ uniquely as a
  finite sum of the form $m = \sum_{k \in \bZ} m_k \tensor x^k$ with
  certain $m_k \in M$ (almost all of which are zero); the associated
  element in~$\ker \epsilon$ is
  \[m - \sigma\epsilon (m) = \sum_{k \in \bZ} m_k \tensor x^k -
  \sum_{k \in \bZ} m_kx^k \tensor 1 \ .\] We want to demonstrate that
  this is in the image of $x \tensor 1 - 1 \tensor x$; it is certainly
  enough to prove this for each individual summand $b_k = m_k \tensor
  x^k - m_kx^k \tensor 1$. This is trivial for $k=0$ as $b_0 = 0$. For
  $k > 0$ we obtain $b_k$ as the image of
  \[-\big( m_k x^{k-1} \tensor 1 + m_k x^{k-2} \tensor x + \ldots +
  m_k \tensor x^{k-1} \big)\] under the map $x \tensor 1 - 1 \tensor
  x$; similarly, $b_{-k}$ is the image of
  \[m_{-k} x^{-k} \tensor x^{-1} + m_{-k} x^{-(k-1)} \tensor x^{-2} +
  \ldots + m_{-k} x^{-1} \tensor x^{-k}\] under the same map. This
  proves exactness of~\eqref{eq:ses}.

  Applying this result in each chain level proves that we have a
  similar exact sequence with $M$~replaced by the chain
  complex~$C$. It follows from Lemma~\ref{lem:cone_coker} that the
  canonical map $\mathrm{Cone}\, (x \tensor 1 - 1 \tensor x) \rTo C$
  is a quasi-isomorphism.
\end{proof}

\section{Proof of Theorem~\ref{thm:main}}

\eqref{item:fd} $\Rightarrow$ \eqref{item:hf} Suppose $C$ is
$R$\nbd-finitely dominated. We can then find a strict $R$\nbd-perfect
complex~$D$ of $R$\nbd-modules, together with mutually inverse
$R$\nbd-linear chain homotopy equivalences $f \colon C \rTo D$ and $g
\colon D \rTo C$. Let $x$~denote the $R$\nbd-linear self map of~$C$
given by ``multiplication by~$x$'', as before. Since the maps $x$ and
$xgf$ are homotopic, there is an isomorphism of
$R[x,x^{-1}]$\nbd-module complexes $T(xgf) \iso T(x)$,
cf.~Lemma~\ref{lem:basic_prop}~(2).  By \textsc{Mather}'s mapping
torus trick Proposition~\ref{prop:weak_Mather} there is an
$R[x,x^{-1}]$\nbd-linear quasi-isomorphism $f_* \colon T(xgf) \rTo
T(fxg)$. Finally, there is a quasi-isomorphism $T(x) \rTo C$, by
Lemma~\ref{lem:C_is_Tx}. We thus have quasi-isomorphisms
\[C \lTo T(x) \lTo T(xgf) \rTo T(fxg) \ .\] Now the chain complex
$T(fxg)$ is strict perfect over~$R_1 = R[x,x^{-1}]$ since $D$~is
strict perfect over~$R$; in addition, its finiteness obstruction is
trivial by Proposition~\ref{cone-no-obstruction}, applied to the
defining mapping cone of the mapping torus, so that $T(fxg)$ is
homotopy equivalent to a bounded complex of finitely generated free
$R[x,x^{-1}]$\nbd-modules.  Moreover, all other chain complexes are
bounded below and consist of projective $R[x,x^{-1}]$\nbd-modules,
hence the quasi-isomorphisms are in fact homotopy equivalences. It
follows that $C$~is homotopy $R_1$\nbd-finite.

We can iterate the argument, replacing $R$~by $R_k$ and $R_1$ by
$R_{k+1}$, proving that $C$ is indeed homotopy $R_j$\nbd-finite for $1
\leq j \leq n$.

This argument works for any re-numbering of the variables in precisely
the same way. We have thus shown that condition~\eqref{item:hf} holds.

\smallskip

\eqref{item:hf} $\Rightarrow$ \eqref{item:acyclicall} For $1 \leq j
\leq n$ there is a bounded complex $D^j$ of finitely generated free
$R_j$\nbd-modules which is homotopy equivalent (over~$R_j$) to~$C$, by
hypothesis. It follows that there are homotopy equivalences
\begin{equation}
  \label{eq:acyclic}
  \begin{aligned}
  C \tensor_{R_j} R_{j-1} ((x_j)) & \simeq D^j \tensor_{R_j} R_{j-1}
  ((x_j)) \qquad \text{and} \\ 
  C \tensor_{R_j} R_{j-1} ((x_j^{-1}))
  & \simeq D^j \tensor_{R_j} R_{j-1} ((x_j^{-1})) \ .
\end{aligned}
\end{equation}
Now we can apply \textsc{Ranicki}'s Theorem~\ref{thm:ranicki}
iteratively to the chain complexes~$D^j$, $1 \leq j \leq n$,
noting that by the previous step (or the hypothesis, for $j=1$) we
know $D^j$ to be $R_{j-1}$\nbd-finitely dominated. It follows that the
chain complexes in~\eqref{eq:acyclic} are acyclic as claimed.

\smallskip

\eqref{item:acyclicall} $\Rightarrow$ \eqref{item:acyclicfixed} is trivial.

\smallskip

\eqref{item:acyclicfixed} $\Rightarrow$ \eqref{item:fd} First we may
assume that $C$~itself is a strict $R_n$\nbd-perfect chain
complex. Since a finitely generated projective module is a direct
summand of a finitely generated free one, there exists a strict
$R_n$\nbd-perfect complex $C^\prime \in \Chb(R_n)$ with trivial
differentials such that $D = C \oplus C^\prime$ consists of finitely
generated free modules.

By algebraic transversality \cite[Proposition~1]{Ranicki-findom} there
exist chain complexes
\[D^+ \in \Chb \big(R_{n-1}[x_n]\big) \ , \quad D^- \in
\Chb\big(R_{n-1}[x_n^{-1}]\big) \quad \text{and} \quad L \in
\Chb(R_{n-1})\] consisting of finitely generated free modules over
their respective rings, together with chain maps forming a
short exact sequence
\begin{equation}
  \label{eq:ses_algtrans}
  0 \rTo L \rTo D^+ \oplus D^- \rTo^{f^+ - f^-} D \rTo 0
\end{equation}
of $R_{n-1}$\nbd-module chain complexes, such that the adjoint maps
\[D^+ \tensor_{R_{n-1}[x_n]} R_n \rTo D \quad \text{and} \quad D^-
\tensor_{R_{n-1}[x_n^{-1}]} R_n \rTo D\] are isomorphisms of
$R_n$\nbd-module chain complexes.

Before going any further we introduce a new piece of notation. Given a
diagram of chain complexes of modules
\[\mathcal{Z} = \ \big( Z^- \rTo^{g^-} Z \lTo^{g^+} Z^+ \big)\]
we define $\Gamma (\mathcal{Z})$ by the rule
\[\Gamma (\mathcal{Z}) = \mathrm{Cone}\, (Z^+ \oplus Z^-
\rTo[l>=3em]^{g^+ - g^-} Z)[-1] \ .\] If all the complexes $Z$, $Z^+$
and~$Z^-$ are concentrated in degree~$0$ then $\Gamma(\mathcal{Z})$
computes derived inverse limits as $H_{-k} \Gamma (\mathcal{Z}) =
\lim{}^k (\mathcal{Z})$; in general, the homology modules
of~$\Gamma(\mathcal{Z})$ should be thought of as hyper-derived inverse
limits. --- Straight from the definition we see that $\Gamma(Z \rTo 0
\lTo 0) = \Gamma (0 \rTo 0 \lTo Z) = Z$. In addition, from the
properties of mapping cones it is clear that a commutative diagram
\begin{diagram}
  Z^- & \rTo & Z & \lTo & Z^+ \\
  \dTo<\simeq && \dTo<\simeq && \dTo<\simeq \\
  Y^- & \rTo & Y & \lTo & Y^+
\end{diagram}
with vertical morphisms all quasi-isomorphisms induces a
quasi-isomorphism
\[\Gamma( Z^-  \rTo  Z  \lTo  Z^+ ) \rTo^\simeq \Gamma( Y^-  \rTo  Y
\lTo  Y^+ ) \ .\]

We return to the actual proof. By Corollary~\ref{cor:ker_cone} the
sequence~\eqref{eq:ses_algtrans} yields a quasi-isomorphism
\begin{equation}
  \begin{split}
    \label{eq:left_finite}
    L \rTo^\simeq \mathrm{Cone}\, \big( D^+ \oplus D^- \rTo^{f^+ - f^-}
    D \big)[-1] \ \ \\ = \Gamma (D^- \rTo D \lTo D^+)
  \end{split}
\end{equation}
which is actually a homotopy equivalence since all constituent chain
complexes consist of projective $R_{n-1}$\nbd-modules.

We will now replace the right-hand side of~\eqref{eq:left_finite} by a
quasi-isomorphic complex which contains the chain complex~$C$ as a
direct summand up to homotopy, thereby proving that $C$~is
$R_{n-1}$\nbd-finitely dominated.  We have a short exact sequence of
$R_{n-1}[x_n]$\nbd-modules {\small
  \[0 \rTo R_{n-1}[x_n] \rTo^{(+,+)} R_{n-1}[[x_n]] \oplus
  R_{n-1}[x_n, x_n^{-1}] \rTo^{(+,-)} R_{n-1} ((x_n)) \rTo 0 \
  ;\]}\relax we thus get, by taking tensor product over~$R_{n-1}[x_n]$
with~$D^+$, a short exact sequence of chain complexes
\[0 \rTo D^+ \rTo^{(+,+)} D^+[[x_n]] \oplus D^+ [x_n, x_n^{-1}]
\rTo^{(+,-)} D^+ ((x_n)) \rTo 0 \ .\] Here we have used the following
abbreviations:
\begin{align*}
  D^+ [[x_n]] &= D^+ \tensor_{R_{n-1}[x_n]} R_{n-1}[[x_n]] \\
  D^+ ((x_n)) &= D^+ \tensor_{R_{n-1}[x_n]} R_{n-1}((x_n)) \\
  D^+[x_n,\, x_n^{-1}] &= D^+ \tensor_{R_{n-1} [x_n]} R_{n-1} [x_n,\, x_n^{-1}] = D^+ \tensor_{R_{n-1}[x_n]} R_n
\end{align*}
Invocation of Corollary~\ref{cor:ker_cone} gives us a
quasi-isomorphism
\begin{equation}
  \label{eq:h1}
  D^+ \rTo^\simeq \Gamma \big( D^+ [x_n, x_n^{-1}] \rTo D^+ ((x_n))
  \lTo D^+ [[x_n]] \big) \ .
\end{equation}

Recall that by construction of~$D^+$ we have isomorphisms $D^+
[x_n,x_n^{-1}] \iso D$ and
\begin{align*}
  D^+ ((x_n)) & \iso D^+ \tensor_{R_{n-1}[x_n]}
  R_{n-1} [x_n, x_n^{-1}] \tensor_{R_{n-1}[x_n, x_n^{-1}]} R_{n-1}
  ((x_n)) \\
  & \iso D \tensor_{R_n} R_{n-1} ((x_n))
\end{align*}
so that~\eqref{eq:h1} becomes the quasi-isomorphism
\[D^+ \rTo^\simeq_{g^+} H^+ := \Gamma \big( D \rTo D \tensor_{R_n}
R_{n-1} ((x_n)) \lTo D^+ [[x_n]] \big) \ .\]
Similarly, by exchanging~$x_n$ and $x_n^{-1}$ we obtain a
quasi-isomorphism
\[D^- \rTo^\simeq_{g^-} H^- := \Gamma \big( D \rTo D \tensor_{R_n}
R_{n-1} ((x_n^{-1})) \lTo D^- [[x_n^{-1}]] \big) \ ,\]
where we have used the notation
\[D^- [[x_n^{-1}]] = D^- \tensor_{R_{n-1}[x_n^{-1}]}
R_{n-1}[[x_n^{-1}]] \ .\]

We have an obvious commutative diagram
\begin{diagram}
  D & \rTo[l>=3em] & D \tensor_{R_n} R_{n-1} ((x_n)) & \lTo[l>=3em] & D^+ [[x_n]] \\
  \dTo && \dTo && \dTo \\
  D & \rTo & 0 & \lTo & 0
\end{diagram}
which upon application of the functor~$\Gamma$ results in a chain
complex map $h^+ \colon H^+ \rTo D$. A similar construction yields the
map $h^- \colon H^- \rTo D$, and these maps fit into another
commutative diagram of chain complexes
\begin{diagram}
  D^- & \rTo^{f^-} & D & \lTo^{f^+} & D^+ \\
  \dTo<\simeq>{g^-} && \dTo<\simeq>\id && \dTo<\simeq>{g^+} \\
  H^- & \rTo^{h^-} & D & \lTo^{h^+} & H^+
\end{diagram}
which results in a quasi-isomorphism
\begin{equation}
  \label{eq:results_qiso}
  \Gamma \big( D^- \rTo^{f^-} D \lTo^{f^+} D^+ \big) \rTo^\simeq
  \Gamma \big( H^- \rTo^{h^-} D \lTo^{h^+} H^+ \big) \ .
\end{equation}

Recall that $D$ splits as $D = C \oplus C^\prime$, and that
consequently the tensor product $ D \tensor_{R_n}
R_{n-1} ((x_n))$ splits as a direct sum of
\[ C \tensor_{R_n} R_{n-1} ((x_n)) \quad \text{and} \quad C^\prime
\tensor_{R_n} R_{n-1} ((x_n)) \ .\] The former summand is acyclic by
our hypothesis \eqref{item:acyclicfixed} (for $j = n$) so that all
vertical maps in the following commutative diagram are
quasi-isomorphisms:{\small
\begin{diagram}[labelstyle=\scriptstyle,loose,w=0pt]
  C \oplus C^\prime & \rTo[l>=3em] & C \tensor_{R_n} R_{n-1} ((x_n&))
  & \oplus \, & C^\prime \tensor_{R_n} R_{n-1} ((x_n)) & \lTo[l>=3em] & D^+ [[x_n]] \\
  \dTo<\simeq>{(\id,\id)} &&&& \dTo<\simeq>{(0,\id)} &&& \dTo<\simeq>\id \\
  C \oplus C^\prime & \rTo[l>=3em] && 0
  & \oplus \, & C^\prime \tensor_{R_n} R_{n-1} ((x_n)) & \lTo[l>=3em] & D^+ [[x_n]]
\end{diagram}}%
That is, by applying the functor $\Gamma$ we obtain a
quasi-isomorphism from~$H^+$ to
\[K^+ := \Gamma \big( C \oplus C^\prime \rTo[l>=3em]^{(0,+)} C^\prime
\tensor_{R_n} R_{n-1} ((x_n)) \lTo[l>=3em] D^+ [[x_n]] \big) \ ,\]
with ``$+$'' indicating the natural map from~$C^\prime$ into the
tensor product. Moreover, the map $h^+ \colon H^+ \rTo D$ factors
through this new map $H^+ \rTo K^+$.
Similarly, $H^-$~is quasi-isomorphic to
\[K^- := \Gamma \big( C \oplus C^\prime \rTo[l>=3.5em]^{(0,+)}
C^\prime \tensor_{R_n} R_{n-1} ((x_n^{-1})) \lTo D^- [[x_n^{-1}]]
\big) \ .\] 

We thus obtain a commutative diagram with vertical quasi-isomorphisms
\begin{diagram}
  H^- & \rTo^{h^-} & D & \lTo^{h^+} & H^+ \\
  \dTo<\simeq && \dTo<\simeq && \dTo<\simeq \\
  K^- & \rTo & C \oplus C^\prime & \lTo & K^+
\end{diagram}
resulting in a quasi-isomorphism from the target
of~\eqref{eq:results_qiso} to
\begin{equation}
  \label{eq:final}
  \Gamma \big( K^- \rTo C \oplus C^\prime \lTo K^+ \big) \ .
\end{equation}
But by direct inspection this last complex contains
\[\Gamma (C \rTo C \lTo C) = \mathrm{Cone}\, \big( C \oplus C
\rTo[l>=3em]^{(+,-)} C \big)[-1] \simeq C\] as a direct summand.
This becomes clear when writing out the definitions of the
constituents of~\eqref{eq:final}, see Fig.~\ref{fig:fig}.
\begin{figure}[!t]
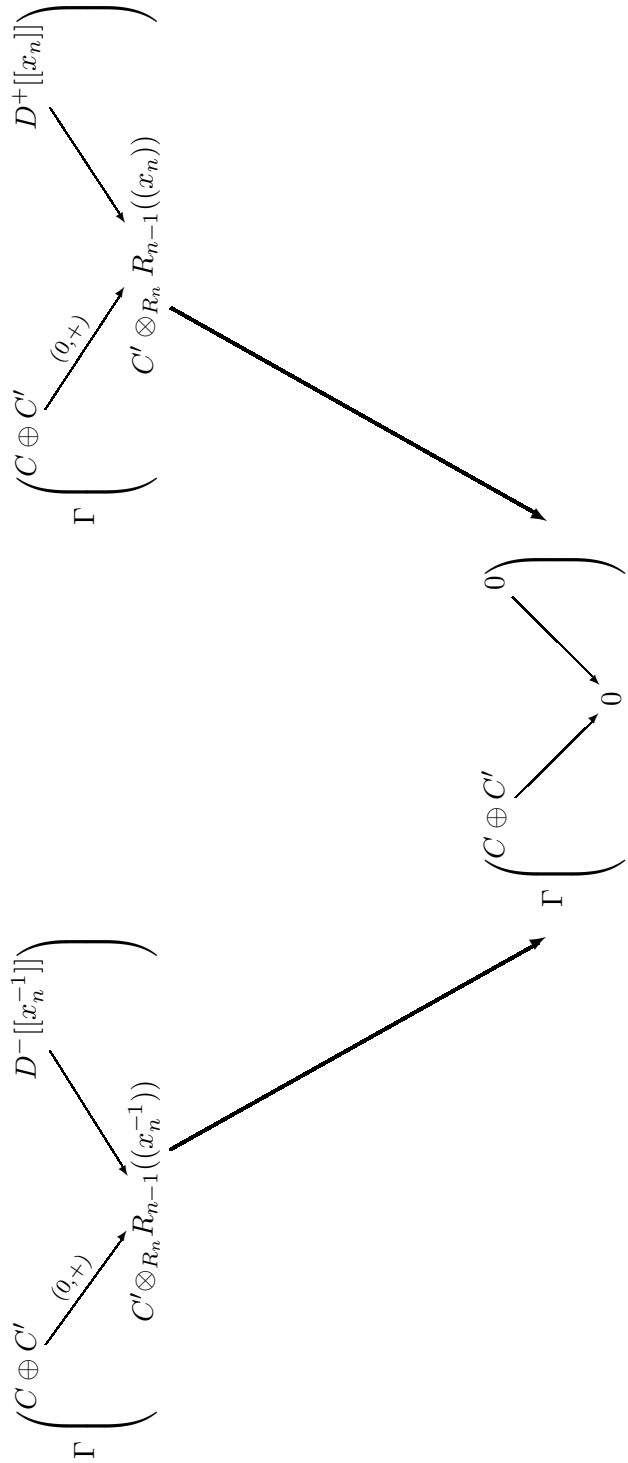

  \begin{diagram}[PS,h=8em,w=7em,landscape,thick]
    \Gamma \left(
      \begin{diagram}[small,objectstyle=\textstyle,labelstyle=\scriptstyle,thin]
        C \oplus C^\prime &&&&  D^- [[x_n^{-1}]] \\
        & \rdTo>{(0,+)} && \ldTo \\
        & C^\prime \tensor_{R_n} & R_{n-1} & ((x_n^{-1})) &
      \end{diagram}
    \right) &&&& \Gamma \left(
      \begin{diagram}[small,objectstyle=\textstyle,labelstyle=\scriptstyle,thin]
        C \oplus C^\prime &&&& D^+ [[x_n]] \\
        & \rdTo>{(0,+)} && \ldTo \\
        && C^\prime \tensor_{R_n} R_{n-1} ((x_n)) &&
      \end{diagram}
    \right) \\
    & \rdTo(1,2) && \ldTo(1,2) & \\
    && \Gamma \left(
      \begin{diagram}[small,objectstyle=\textstyle,labelstyle=\scriptstyle,thin]
        C \oplus C^\prime &&&& 0 \\
        & \rdTo && \ldTo & \\
        && 0 &&
      \end{diagram}
    \right) &&
  \end{diagram}%
  \caption{An expanded version of~\eqref{eq:final}. The thick diagonal
  maps are induced by the obvious maps of small diagrams.}
  \label{fig:fig}
\end{figure}
Indeed, the two ``outer'' summands~$C$ only map non-trivially to the
``inner'' summand~$C$, with the latter not receiving any other
non-trivial map, so that $\Gamma (C \rTo C \lTo C)$ appears as a
direct summand once $\Gamma$ is applied to the diagram in
Fig.~\ref{fig:fig}.

It follows that $C$~is homotopy equivalent to a summand of the chain
complex~\eqref{eq:final} which is quasi-isomorphic,
via~\eqref{eq:results_qiso} and~\eqref{eq:left_finite}, to the finite
complex~$L$ of $R_{n-1}$\nbd-modules. Consequently~$C$, considered as
an $R_{n-1}$\nbd-module complex, is a retract up to homotopy of the
chain complex~$L$. Indeed, the complex~\eqref{eq:final} can be
replaced, up to quasi-isomorphism, by a bounded below complex of
projective $R_{n-1}$\nbd-modules which is quasi-isomorphic, and hence
chain homotopy equivalent, to~$L$. Using the fact that $C$~is a
bounded below complex of projective $R_{n-1}$\nbd-modules as well it
is then standard homological algebra to construct the desired maps of
complexes $\alpha \colon C \rTo L$ and $\beta \colon L \rTo C$
together with a chain homotopy $\beta \alpha \simeq \id$. It now
follows from \cite[Proposition~3.2~(ii)]{Ranicki:algfin} that $C$~is
$R_{n-1}$\nbd-finitely dominated.

In case $n=1$ this finishes the proof of \eqref{item:acyclicfixed}
$\Rightarrow$ \eqref{item:fd}.  For $n>1$ we observe that $C$ is now
homotopy equivalent over $R_{n-1}$ to a strict perfect complex $B \in
\Chb(R_{n-1})$ which satisfies condition~\eqref{item:acyclicfixed} of
the Theorem for $j<n$.  By induction, $B$~is $R$\nbd-finitely
dominated, hence so is~$C$. --- This finishes the proof for
general~$n$.

\section{A non-trivial finitely dominated chain complex}

We will now discuss a generalisation of a one\nbd-variable example
given by \textsc{Hughes} and \textsc{Ranicki}
\cite[Example~23.19]{Hughes-Ranicki}. This serves to illustrate the
existence of non-trivial finitely dominated chain complexes.

Let $R$ be a commutative integral domain, and write $R_n$ for the
\textsc{Laurent} polynomial ring in indeterminates $x_1,\, x_2,\,
\cdots,\, x_n$ as before. We actually restrict to the case $n=2$,
leaving the easy generalisation for higher~$n$ to the reader.
Consider the following square diagram:
\begin{diagram}[LaTeXeqno]
\label{diag:sq2}
  R_2 & \rTo^{1-x_1x_2} & R_2 \\ \dTo<{1-x_1} && \dTo<{1-x_1} \\
  R_2 & \rTo^{1-x_1x_2} & R_2
\end{diagram}
Let $h \colon D \rTo D$ denote the chain complex obtained by taking
mapping cones in vertical direction, and let $C$ be the mapping cone
of~$h$.

Clearly the complex $C$~is not acyclic; indeed, the element $x_2 \in
R_2$~represents a non-trivial element in the bottom homology
of~$C$. However, we claim that the four chain complexes
\begin{align*}
  C \tensor_{R_1} R((x_1)) & \quad \text{and} \quad C \tensor_{R_1} R((x_1^{-1})) \ , \\
  C \tensor_{R_2} R_1((x_2)) & \quad \text{and} \quad C \tensor_{R_2} R_1 ((x_2^{-1})) 
\end{align*}
are all acyclic. This can be seen as follows: First, the vertical maps
in the square~\eqref{diag:sq2} become isomorphisms after tensoring
over~$R_1$ with $R_1((x_1))$ as $1-x_1$ is a unit in the latter
ring. Consequently, by tensoring and taking mapping cones in vertical
directions we obtain a map of acyclic chain complexes
\[\mathrm{Cone}\, \big( R_2 \rTo[l>=3em]^{1-x_1} R_2 \big) \rTo \mathrm{Cone}\,
\big( R_2 \rTo[l>=3em]^{1-x_1} R_2 \big)\] whose mapping cone~$K$ is
acyclic as well. But formation of mapping cones is compatible with
taking tensor products so that there is an isomorphism $K \iso C
\tensor_{R_1} R((x_1))$. Consequently the latter chain complex is
acyclic. The same argument with the roles of $x_1$ and $x_1^{-1}$
reversed proves that $C \tensor_{R_1} ((x_1^{-1}))$ is acyclic as
well. --- Tensoring the square diagram~\eqref{diag:sq2} over~$R_2$
with $R_1((x_2))$ and taking mapping cones in vertical directions
results in a chain complex map $g \colon E \rTo E$ whose mapping
cone~$J$ is isomorphic to $C \tensor_{R_2} R_1((x_2))$. Now as a map
of graded modules (\ie, disregarding differentials) the map~$g$ is
given by the map
\begin{multline*}
  R_2[1] \tensor_{R_2} R_1((x_1)) \,\oplus\, R_2 \tensor_{R_2}
  R_1((x_1)) \\ \rTo R_2[1] \tensor_{R_2} R_1((x_1)) \,\oplus\, R_2
  \tensor_{R_2} R_1((x_1))
\end{multline*}
induced by multiplication by $1-x_1x_2$. But this polynomial is a unit
in the ring~$R_1((x_2))$ (as $x_1$~is a unit in~$R_1$) so that $g$~is
in fact an isomorphism of chain complexes. It follows that $C
\tensor_{R_2} R_1((x_2)) \iso \mathrm{Cone}\, (g)$ is acyclic. By
exchanging $x_2$ and~$x_2^{-1}$ we see that $C \tensor_{R_2}
R_1((x_2^{-1}))$ is acyclic as well.

By Theorem~\ref{thm:main} this shows that the complex~$C$ is
$R$-finitely dominated. The Theorem also says that the chain complexes
\[C \tensor_{R[x_2, x_2^{-1}]} R((x_2)) \quad \text{and} \quad C
\tensor_{R_2} R[x_2,\, x_2^{-1}] ((x_1))\] are acyclic, but note that
this cannot be proved as easily as above ({\it viz.}, by showing that
the horizontal or vertical maps of~\eqref{diag:sq2} become
isomorphisms after application of a tensor product functor). It
appears that the freedom to re-number the variables is relevant for
detecting finite domination in practise.

\section{Finite domination over fields}

We finish the paper by discussing finite domination over fields which
is (not surprisingly) much simpler than the general case. Suppose $F$
is a field, and $C$ is a bounded chain complex of finitely generated
projective modules over the \textsc{Laurent} ring
\[L = F[z_1,\, z_1^{-1},\, z_2,\, z_2^{-1},\, \cdots,\, z_n,\,
z_n^{-1}] \ .\] Since $F$~is a field, $C$~is $F$-finitely dominated if
and only if $\dim_F H_k C < \infty$ for all~$k$. (See
\cite[Theorem~1.7.13]{Rosenberg} for a proof covering the more general
situation of a \textsc{noether}ian ground ring. Since there is no
difference between free and projective $F$\nbd-modules, $C$~is
$F$\nbd-finitely dominated if and only if $C$~is $F$\nbd-homotopy
finite.) We obtain the following multi-variable version of \cite[\S5,
Example]{Ranicki-findom}:

\begin{theorem}
  \label{thm:field}
  The complex~$C$ is $F$-finitely dominated if and only if the induced
  chain complexes
  \[C \tensor_{F[z_j, z_j^{-1}]} F(z_j) \ , \quad j = 1,\, 2,\,
  \cdots,\, n \ ,\] are acyclic. (Here $F(z_j)$ denotes the field of
  rational functions in~$z_j$.)
\end{theorem}

\begin{proof}
  Suppose first that $C$ is finitely dominated. 
  For fixed~$k$ and~$j$, the multiplication action of~$z_j$ on~$C$
  determines an endomorphism $f_j$ of the finite-dimensional
  $F$-vector space $H_k (C)$. Its characteristic polynomial $p_j (x) =
  \det (f_j - x \cdot \id)$ satisfies $p_j (f_j) = 0$, by
  \textsc{Cayley}-\textsc{Hamilton}. Note that as a self map of
  $H_k(C)$, the action of $p_j(f_j)$ coincides with the one given by
  multiplication with the polynomial $p_j(z_j)$. For any primitive
  tensor $a \tensor b \in H_k(C) \tensor_{F[z_j, z_j^{-1}]} F(z_j)$ we
  have the chain of equalities
  \[a \tensor b = a \tensor (p_j \cdot b/p_j) = (a \cdot p_j) \tensor
  (b/p_j) = 0 \tensor (b/p_j) = 0\] so that $H_k (C) \tensor_{F[z_j,
    z_j^{-1}]} F(z_j) = 0$. But $F(z_j)$ is a localisation of
  $F[z_j,\, z_j^{-1}]$ ({\it viz.}, its quotient field), whence we
  have an isomorphism
  \[H_k \big(C \tensor_{F[z_j, z_j^{-1}]} F(z_j)\big) \cong H_k (C)
  \tensor_{F[z_j, z_j^{-1}]} F(z_j) = 0 \ .\] This proves that $C
  \tensor_{F[z_j, z_j^{-1}]} F(z_j)$ is acyclic as claimed.

  \medbreak

  To prove the converse suppose that $C \tensor_{F[z_j, z_j^{-1}]}
  F(z_j)$ is acyclic for all~$j$. Fix $k$ and~$j$.  Exactness of
  localisation allows us to rewrite this hypothesis as
  \[H_k(C) \tensor_{F[z_j, z_j^{-1}]} F(z_j) \cong H_k \big(C
  \tensor_{F[z_j, z_j^{-1}]} F(z_j) \big) = 0 \ .\] This implies that
  the image of any element $g \in H_k(C)$ in the tensor product
  $H_k(C) \tensor_{F[z_j, z_j^{-1}]} F(z_j)$ is trivial. As an
  \textsc{abel}ian group said tensor product is a quotient of $H_k(C)
  \tensor_\bZ F(z_j)$ by relations of the form $a \tensor_\bZ (pb) -
  (ap) \tensor_\bZ b$, for $p \in F[z_j,\, z_j^{-1}]$. In other words
  we find finitely many \textsc{Laurent} polynomials $p_i \in F[z_j,\,
  z_j^{-1}]$, and elements $a_i \in H_k(C)$ and $b_i \in F(z_j)$, all
  depending on~$g$, such that
  \[g \tensor_\bZ 1 = \sum_i \big( a_i \tensor_\bZ (p_i b_i) - (a_i
  p_i) \tensor_\bZ b_i \big) \ . \tag{*}\] Since $F(z_j)$ is the
  quotient field of $F[z_j,\, z_j^{-1}]$ we find a \textsc{Laurent}
  polynomial~$p(g)$, depending on~$g$, such that $b_i p(g) \in
  F[z_j,\, z_j^{-1}]$.

  The ring $L$ is \textsc{noether}ian so that $H_k(C)$ is a finitely
  generated $L$-module. Let $g_1, \, g_2,\, \cdots,\, g_m$ be a set of
  generators, and let $q_j = \prod_{\ell = 1}^m p(g_\ell)$ be the
  product of the \textsc{Laurent} polynomials~$p(g)$ constructed above
  from~(*), where $g$ is replaced in turn by the~$g_\ell$. Then, using
  the right $F[z_j,\, z_j^{-1}]$-module structure on~$F(z_j)$,
  equation~(*) for $g = g_\ell$ says that
  \[g_\ell \tensor_\bZ q_j = \sum_i \big( a_i \tensor_\bZ (p_i b_i
  q_j) - (a_i p_i) \tensor_\bZ (b_i q_j) \big) \ .\] By choice
  of~$q_j$ we have $b_i q_j \in F[z_j,\, z_j^{-1}]$ so that
  consequently
  \[g_\ell q_j \tensor_{[z_j, z_j^{-1}]} 1 = g_\ell \tensor_{[z_j,
    z_j^{-1}]} q_j = 0 \quad \text{in } H_k(C) \tensor_{F[z_j,
    z_j^{-1}]} F[z_j,\, z_j^{-1}] \cong H_k (C) \ ,\] that is, $g_\ell
  q_j = 0 \in H_k(C)$. Since the $g_\ell$ generate~$H_k(C)$ this
  implies that multiplication by~$q_j$ annihilates~$H_k(C)$.

  By what we have just shown, $H_k(C)$ is an $L/(q_1,\, q_2,\,
  \cdots,\, q_n)$-module in a natural way. But $H_k(C)$ is finitely
  generated as an $L$-module, hence as a module over the quotient
  $L/(q_1,\, q_2,\, \cdots,\, q_n)$ which in turn is a finite
  dimensional $F$-vector space. It follows that $H_k(C)$ is of finite
  dimension over~$F$ as required.
\end{proof}


\providecommand{\bysame}{\leavevmode\hbox to3em{\hrulefill}\thinspace}
\providecommand{\MR}{\relax\ifhmode\unskip\space\fi MR }
\providecommand{\MRhref}[2]{%
  \href{http://www.ams.org/mathscinet-getitem?mr=#1}{#2}
}
\providecommand{\href}[2]{#2}

\raggedright

\end{document}